\documentclass[amsthm,11pt]{elsart}
\usepackage{amssymb}
\usepackage{bm}

\usepackage{epsfig}
\usepackage{mathdots}

\usepackage{natbib}
\usepackage{amsmath}

 \theoremstyle{plain}
\newtheorem{theorem}{Theorem}
\newtheorem{corollary}[theorem]{Corollary}

\newtheorem{lemma}[theorem]{Lemma}
\newtheorem{proposition}[theorem]{Proposition}
\theoremstyle{definition}
\newtheorem{definition}{Definition}
\theoremstyle{remark}
\newtheorem{remark}[theorem]{Remark}

\numberwithin{equation}{section}
\numberwithin{theorem}{section}

\newcommand{\bT}{\begin{theorem}}
\newcommand{\eT}{\end{theorem}}
\newcommand{\bProp}{\begin{proposition}}
\newcommand{\eProp}{\end{proposition}}
\newcommand{\bE}{\begin{example}}
\newcommand{\eE}{\end{example}}
\newcommand{\bL}{\begin{lemma}}
\newcommand{\eL}{\end{lemma}}
\newcommand{\bP}{\begin{proof}}
\newcommand{\eP}{\end{proof}}
\newcommand{\bC}{\begin{corollary}}
\newcommand{\eC}{\end{corollary}}
\newcommand{\bD}{\begin{definition}}
\newcommand{\eD}{\end{definition}}
\newcommand{\be}{\begin{enumerate}}
\newcommand{\ee}{\end{enumerate}}
\newcommand{\beqa}{\begin{eqnarray*}}
\newcommand{\eeqa}{\end{eqnarray*}}
\newcommand{\beqaa}{\begin{eqnarray}}
\newcommand{\eeqaa}{\end{eqnarray}}
\newcommand{\ba}{\begin{array}}
\newcommand{\ea}{\end{array}}

\newcommand{\PHI}[2]{\left._{#1}\phi \right._{#2}}

\newcommand{\phiargs}[4]{\left[ \ba{l} #1 \\ #2 \ea; #3 ,#4\right]}
\newcommand{\wpalpha}{\bm{\alpha}}
\newcommand{\wpbeta}{\bm{\beta}}

\setbox0=\hbox{$+$}
\newdimen\plusheight
\plusheight=\ht0
\def\+{\;\lower\plusheight\hbox{$+$}\;}

\setbox0=\hbox{$-$}
\newdimen\minusheight
\minusheight=\ht0
\def\-{\;\lower\minusheight\hbox{$-$}\;}

\setbox0=\hbox{$\cdots$}
\newdimen\cdotsheight
\cdotsheight=\plusheight
\def\cds{\lower\cdotsheight\hbox{$\cdots$}}
\begin{document}
\begin{frontmatter}

\title{Lifting Bailey Pairs to WP-Bailey Pairs}
    \author{James McLaughlin}
\address{Department of Mathematics, 124 Anderson Hall, \\West Chester University, West Chester PA 19383}
\ead{jmclaughl@wcupa.edu}

\author{Andrew V. Sills}
\address{Department of Mathematical Sciences, 203 Georgia Avenue Room 3008,\\ Georgia Southern University, Statesboro, GA 30460-8093}
\ead{ASills@GeorgiaSouthern.edu}

\author{Peter Zimmer}
\address{Department of Mathematics, 124 Anderson Hall, \\West Chester University, West Chester PA 19383}
\ead{pzimmer@wcupa.edu}



\date{\today}

\begin{abstract}
A pair of sequences $(\wpalpha_{n}(a,k,q),\wpbeta_{n}(a,k,q))$ such
that\\ $\wpalpha_0(a,k,q)=1$ and
\[
\wpbeta_{n}(a,k,q) = \sum_{j=0}^{n} \frac{(k/a; q)_{n-j}(k;
q)_{n+j}}{(q;q)_{n-j}(aq;q)_{n+j}}\wpalpha_{j}(a,k,q)
\]
is termed a \emph{WP-Bailey Pair}. Upon setting $k=0$ in such a pair
we obtain a \emph{Bailey pair}.

In the present paper we consider the problem of ``lifting" a Bailey
pair to a WP-Bailey pair, and use some of the new WP-Bailey pairs
found in this way to derive some new identities between basic
hypergeometric series and new single sum- and double sum identities
of the Rogers-Ramanujan-Slater type.
\end{abstract}

\end{frontmatter}


\section{Introduction}

A pair of sequences $\big( \alpha_n (a,q) ,  \beta_n (a,q) \big)$ that satisfy
$\alpha_0 (a,q) = 1$ and
\begin{multline}
\beta_n(a,q) = \sum_{r=0}^n \frac{ \alpha_r (a,q)  }{ (q;q)_{n-r} (aq;q)_{n+r} }
  \\ \label{bpeq}
  = \frac{1}{(aq,q;q)_n} \sum_{r=0}^n \frac{ (q^{-n};q)_r }{ (aq^{n+1};q)_r }
  (-1)^r q^{nr- r(r-1)/2} \alpha_r (a,q)
  \end{multline}
where \begin{align*}
           (a;q)_n &:= (1-a)(1-aq)\cdots (1-aq^{n-1}), \\
          (a_1, a_2, \dots, a_j; q)_n &:= (a_1;q)_n (a_2;q)_n \cdots (a_j;q)_n ,\\
           (a;q)_\infty &:= (1-a)(1-aq)(1-aq^2)\cdots, \mbox{ and }\\
          (a_1, a_2, \dots, a_j; q)_\infty &:= (a_1;q)_\infty (a_2;q)_\infty \cdots (a_j;q)_\infty,
          \end{align*}
is termed a \emph{Bailey pair relative to $a$}. Bailey \cite{B47,
B49} showed that, for such a pair, {\allowdisplaybreaks
\begin{multline}\label{Baileyeq}
\sum_{n=0}^{\infty} (y,z;q)_{n}\left ( \frac{aq}{yz}\right )^{n}
\beta_n (a,q)\\= \frac{(aq/y,aq /z;q)_{\infty}}{ (aq, aq/yz;q)_{\infty}}
\sum_{n=0}^{\infty} \frac{(y,z;q)_{n}}{(aq/y,aq/z;q)_n}\left (
\frac{aq}{yz}\right )^{n} \alpha_n (a,q).
\end{multline}
}
 Slater, in
\cite{S51} and \cite{S52}, subsequently used this transformation
of Bailey to derive 130 identities of the Rogers-Ramanujan type.
Slater's method involved specializing $y$ and $z$ so that the
series on right side of \eqref{Baileyeq} became summable, using
the Jacobi triple product identity.
\begin{equation}\label{JTP}
\sum_{n=-\infty}^{\infty} x^nq^{n^2}=(-q/x,-qx,q^2;q^2)_{\infty}.
\end{equation}

 In
\cite{A01}, Andrews extended the definition of a Bailey pair by
setting\\ $\wpalpha_0(a,k,q)=1$ and {\allowdisplaybreaks
\begin{align}\label{WPpair}
{\wpbeta}_{n}(a,k,q) &= \sum_{j=0}^{n}
\frac{(k/a;q)_{n-j}(k;q)_{n+j}}{(q;q)_{n-j}(aq;q)_{n+j}}{\wpalpha}_{j}(a,k,q)\\
&= \frac{(k/a,k;q)_n}{(aq,q;q)_n}\sum_{j=0}^{n}
\frac{(q^{-n};q)_{j}(kq^n;q)_{j}}{(aq^{1-n}/k;q)_{j}(aq^{n+1};q)_{j}}
\left(\frac{qa}{k}\right)^j\wpalpha_{j}(a,k,q). \notag
\end{align}
} Such a pair $(\wpalpha_{n}(a,k,q),\,\wpbeta_{n}(a,k,q))$ is termed a
\emph{WP-Bailey pair}. Examples of WP Bailey pairs were previously given
by Bressoud \cite{B81a} and Singh \cite{S94}. Note that setting
$k=0$ in a WP-Bailey pair generates a standard Bailey pair, but it
is not necessarily true that all standard Bailey pairs can be
derived in this way (at least not if we insist that the
$\wpbeta_{n}(a,k,q)$ in a WP-Bailey pair be in closed form).

We say that the Bailey pair $(\alpha_n (a,q) , \beta_n(a,q) )$
relative to $a$ \emph{lifts} to the WP-Bailey pair
$(\wpalpha_n(a,k,q), \wpbeta_n(a,k,q))$, or equivalently, that\\
$(\wpalpha_n(a,k,q), \wpbeta_n(a,k,q))$ is a \emph{lift} of the pair
$(\alpha_n(a,q) , \beta_n(a,q) )$, if
\[
\wpalpha_n(a,0,q)=\alpha_n (a,q) , \hspace{25pt}\wpbeta_n(a,0,q)=\beta_n (a,q),
\hspace{25pt}\forall\, n \geq 0.
\]

\begin{remark}
Sometimes it will be convenient to suppress one or more of the parameters
$a$, $k$, or $q$ in the notation of ordinary and WP Bailey pairs.  We shall,
however, always distinguish between ordinary and WP Bailey pairs by
denoting the latter in boldface.
\end{remark}

 The following (using slightly different notation) was proved in \cite{McLZ07b}.
\begin{theorem}\label{fbt1}
Let $N$ be a positive integer. Suppose that $\wpalpha_0 =1$ and the
sequences $\{\wpalpha_n\}$ and $\{ \wpbeta_n\}$ are related by
\begin{equation*} \wpbeta_{n} = \sum_{j=0}^{n}
\frac{(k/a;q)_{n-j}(k;q)_{n+j}}{(q;q)_{n-j}(aq;q)_{n+j}}\,\wpalpha_{j}.
\end{equation*}
Then {\allowdisplaybreaks
\begin{multline}\label{fbt1eq}
\sum_{n=0}^{N} \frac{(1-kq^{2n})(y,z,k a q^{N+1}/yz,q^{-N};q)_n
}{(1-k)(k q/y,k q/z,y z q^{-N}/a,k q^{1+N};q)_n}\,q^n \wpbeta_n\\=
\frac{(q k,q k/y z,q a/y,q a/z;q)_{N}} {(q k/y,q k/z,q a,q a/y z;q)_{N}}
\phantom{asdadasdasdabvvbvmvmbnvbnvdasdasd}\\
\times \sum_{n=0}^{N} \frac{(y,z, k a q^{N+1}/y z,q^{-N};q)_{n}}
{(a q/y,a q/z,a q^{1+N}, y z q^{-N}/k;q)_{n}}\left ( \frac{a
q}{k}\right)^n \wpalpha_n.
\end{multline}
}
\end{theorem}
This turned out to be a re-formulation of one of the constructions
Andrews \cite{A01} used to generate the \emph{WP-Bailey lattice},
but we were unaware of the connection initially. Upon letting $N \to
\infty$ we get the following result from \cite{McLZ07a}.
\begin{theorem}\label{6t1}
Subject to suitable convergence conditions, if
\begin{equation}\label{6betaneq}
\wpbeta_n=
\sum_{r=0}^{n}\frac{(k/a;q)_{n-r}}{(q;q)_{n-r}}\frac{(k;q)_{n+r}}
{(a q;q)_{n+r}}\wpalpha_{r},
\end{equation}
then {\allowdisplaybreaks
\begin{multline}\label{wpeq}
\sum_{n=0}^{\infty} \frac{(q\sqrt{k},-q\sqrt{k}, y,z;q)_{n}}
{(\sqrt{k},-\sqrt{k}, q k/y,q k/z;q)_{n}}\left( \frac{q a}{y z }\right )^{n} \wpbeta_n =\\
\frac{(q k,q k/yz,q a/y,q a/z;q)_{\infty}} {(q k/y,q k/z,q a,q
a/yz;q)_{\infty}} \sum_{n=0}^{\infty}\frac{(y,z;q)_{n}}{(q a/y ,q
a/z;q)_n}\left (\frac{q a}{y z}\right)^{n}\wpalpha_n.
\end{multline}
}
\end{theorem}

Notice that, if $\wpalpha_n$ above is independent of $k$, then the
series on the right sides of \eqref{Baileyeq}
 and \eqref{wpeq}  are identical.
 Now suppose that a  standard
Bailey pair as in~\eqref{bpeq} lifts to a WP-Bailey pair in which
$\wpalpha_n$  is independent of $k$. If the standard Bailey pair gives
rise to an identity of the Rogers-Ramanujan-Slater type, for certain
choices of the parameters $y$ and $z$, then it follows that the same
choices for $y$ and $z$ will lead to a generalization of that
identity, since the only occurrence of $k$ on the right side of
\eqref{wpeq} is in the infinite product and the left side of
\eqref{wpeq} will thus also be an infinite product.

In \cite{McLZ07b} we found a WP-Bailey pair that is a lift of
Slater's pair \textbf{F3}.

\begin{theorem}
Define
\begin{align}\label{newwp}
\wpalpha_{n}(1,k)&=
\begin{cases}
1,&n=0,\\
q^{-n/2}+q^{n/2},&n\geq1,
\end{cases}\\
\wpbeta_{n}(1,k)&=
\frac{(k\sqrt{q},k;q)_{n}}{(\sqrt{q},q;q)_{n}}q^{-n/2}. \notag
\end{align}
Then $(\wpalpha_{n}(1,k),\wpbeta_{n}(1,k))$ satisfy \eqref{WPpair}
(with $a=1$).
\end{theorem}

The substitution of this pair into Theorem \ref{fbt1} leads to the
following corollary.
\begin{corollary}
\begin{multline}\label{q1/2pair}
 _{8} \phi _{7} \left [
\begin{matrix}
k,q\sqrt{k}, -q\sqrt{k},y,z,
k\sqrt{q}, kq^{1+N}/yz,q^{-N}\\
\sqrt{ k},-\sqrt{k}, q k/y,q k/z,\sqrt{q},k q^{1+N}, y z q^{-N}
\end{matrix}
; q,\sqrt{q} \right ]\\=
\frac{(qk,qk/yz,q/y,q/z;q)_{N}} {(qk/y,qk/z,q,q/yz;q)_{N}}\\
\times \left( 1+\sum_{n=0}^{N}\frac{(1+q^{n})(y,z,
kq^{1+N}/yz,q^{-N};q)_{n}}{(q/y,q/z, q^{1+N},y z
q^{-N}/k;q)_n}\left (\frac{\sqrt{q} }{k}\right)^{n} \right).
\end{multline}
\begin{equation}\label{q1/2pairpr}
\sum_{n=0}^{\infty} \frac{ (1-k q^{4n})(k;q)_{2n}q^{2n^2-n} }{
(1-k)(q;q)_{2n}}= \frac{(kq^2;q^2)_{\infty}}{(q;q^2)_{\infty} }.
\end{equation}
\begin{equation}\label{q1/2pairpr2}
\sum_{n=0}^{\infty} \frac{ (1-k q^{2n})(k;q)_{n}(-1)^nq^{n(n-1)/2}
}{ (1-k)(q;q)_{n}}= 0.
\end{equation}
\begin{equation}\label{q1/2pairpr3}
\sum_{n=0}^{\infty} \frac{ (1-k
q^{4n})(-q;q^2)_n(k;q)_{2n}q^{n^2-n} }{
(1-k)(-kq;q^2)_n(q;q)_{2n}}=
\frac{(kq^2,-1;q^2)_{\infty}}{(-kq,q;q^2)_{\infty} }.
\end{equation}
\end{corollary}
\begin{proof}
The identity at \eqref{q1/2pair} is immediate, while
\eqref{q1/2pairpr} follows upon letting $N \to \infty$, replacing
$q$ by $q^2$, letting $y, z \to \infty$ and finally using
\eqref{JTP} to sum the right side. The identity at
\eqref{q1/2pairpr2} is a consequence of letting $N \to \infty$,
setting $y=\sqrt{q}$, letting $z \to \infty$ and again using
\eqref{JTP} to sum the right side, and \eqref{q1/2pairpr3} follows
similarly, except we set $y=-\sqrt{q}$ instead.
\end{proof}
Of course the last three identities are not new, as all can easily
be seen to arise as special cases of \eqref{6phi5eq}. However, they
do illustrate how a lift of a standard Bailey pair leads to
generalizations of identities arising from this standard pair (the
identities given by setting $k=0$ in the corollary above).

The discovery of the WP-Bailey pair at \eqref{newwp} which is a lift
of Slater's Bailey pair \textbf{F3}, together with the observation
following Theorem \ref{6t1}, motivated us to investigate if other
standard Bailey pairs could be lifted to a WP-Bailey pair, and to
see what new transformations of basic hypergeometric series and what
new identities of the Rogers-Ramanujan-Slater type would follow from
these new WP-Bailey pairs.

It turned out that several of the lifts of Bailey pairs that we
found could be derived as special cases of a result (see
\eqref{wpAB} below) of Singh \cite{S94} (see also Andrews and
Berkovich \cite{AB02}) . However, several others were not so
easily explained, and in attempting to prove that some of these
pairs that were found experimentally were indeed WP-Bailey pairs,
we were led to consider various elementary ways of deriving new
WP-Bailey pairs from existing pairs (ways that are different from
those described by Andrews in \cite{A01}).

We also describe various ways in which double-sum identities of the
Rogers -Ramanujan type identities may be easily derived from
WP-Bailey pairs.

The second author defined three ``multiparameter Bailey pairs"
in~\cite{S07}, which, together with certain families of
$q$-difference equations, ``explain" more than half of the 130
Rogers-Ramanujan type identities in Slater's paper~\cite{S52}. These
multiparameter Bailey lift in a natural way to WP Bailey pairs,
which in turn easily yield a variety of single and double-sum WP
generalizations of Rogers-Ramanujan type identities.

\section{WP-Bailey pairs arising from standard Bailey pairs}
We first tried inserting the $\alpha_n$ that were part of Bailey
pairs found by Slater \cite{S51, S52} into \eqref{WPpair}, and
checking experimentally if the resulting $\beta_n(a,k)$ had closed
forms. As a result, the following WP-Bailey pairs were found. The
letter-number combination (e.g. \textbf{E7'}) refers to the standard
pair in Slater's papers \cite{S51} and \cite{S52}  recovered by
setting $k=0$. In all cases it is to be understood that
$\wpalpha_0=\wpbeta_0=1$.
\begin{align}\label{E7'ab}
\wpalpha_n(q,k)&=\frac{(-1)^n (q^{-n}-q^{n+1})}{(1-q)},&&\phantom{asdadadsadaasdadsdsdasd}(\textbf{E7'})& \\
\wpbeta_n(q,k)&=\frac{(-1)^n(k^2;q^2)_n}{q^n(q^2;q^2)_n}. \notag
\end{align}
\begin{align}\label{F3'ab}
\wpalpha_n(1,k)&=q^{-n/2}+q^{n/2}, &&\phantom{asdadadsadasdasasdsdasd}(\textbf{F3'})&\\
\wpbeta_n(1,k)&=\frac{(k, kq^{1/2};q)_n}{(q^{1/2},q;q)_n}\,q^{-n/2}.
\notag
\end{align}
\begin{align}\label{F4'ab}
\wpalpha_n(q,k)&=\frac{q^{-n/2}+q^{n/2+1/2}}{1+q^{1/2}},  &&\phantom{asadadadsasasadasadasd}(\textbf{F4'})&\\
\wpbeta_n(q,k)&=\frac{(k, k q^{-1/2};q)_n}{(q^{3/2},q;q)_n}\,q^{-n/2}.
\notag
\end{align}
\begin{align}\label{H3'ab}
\wpalpha_n(1,k)&=(-1)^n q^{-n^2/2}\left(q^{-3 n/2}+q^{3 n/2}\right), &&\phantom{asdadasddaasd}(\textbf{H3'})&\\
\wpbeta_n(1,k)&=\frac{(-1)^n (1-kq^n+kq^{2n})
(k;q)_{n}}{q^{(n^2+3n)/2}(q;q)_n}. \notag
\end{align}
\begin{align}\label{H4'ab}
\wpalpha_n(1,k)&=(-1)^n q^{-n^2/2}\left(q^{-n/2}+q^{n/2}\right),&&\phantom{asdadadsaasdasd}(\textbf{H4'})&\\
\wpbeta_n(1,k)&=\frac{(-1)^n (k;q)_{n}}{q^{(n^2+n)/2}(q;q)_n}. \notag
\end{align}
\begin{align}\label{H5'ab}
\wpalpha_n(1,k)&=q^{-n}-q^{n}, &&\phantom{asdadadsadadasd}(\textbf{H5'})&\\
\wpbeta_n(1,k)&=\frac{ (k,k;q)_{n}(1-2 k q^n + k q^{2 n})}
{(1-k)q^{n}(q,q;q)_n}. \notag
\end{align}
\begin{align}\label{H6'ab}
\wpalpha_n(1,k)&=0, &&\phantom{asdadadsadaasdadsdaassasddsdasd}(\textbf{H6'})&\\
\wpbeta_n(1,k)&=\frac{ (k,k;q)_{n}}{(q,q;q)_n}. \notag
\end{align}
\begin{align}\label{H7'ab}
\wpalpha_n(1,k)&=2(-1)^n, &&\phantom{asdadadsadaasdadssdsasdasd}(\textbf{H7'})&\\
\wpbeta_n(1,k)&=(-1)^n\,\frac{ (k^2;q^2)_{n}}{(q^2;q^2)_n}. \notag
\end{align}
\begin{align}\label{H8'ab}
\wpalpha_n(1,k)&=(-1)^n(q^{-n}+q^{n}), &&\phantom{asdadadsadaasddasd}(\textbf{H8'})&\\
\wpbeta_n(1,k)&=(-1)^n\frac{ (k^2;q^2)_{n} (1+kq^{2n})}
{(1+k)q^n(q^2;q^2)_n}. \notag
\end{align}
\begin{align}\label{H12'ab}
\wpalpha_n(q,k)&=q^{-n}-q^{n+1}, &&\phantom{ad}(\textbf{H12'})&\\
\wpbeta_n(q,k)&=\frac{ (k;q)_{n} (k;q)_{n-1}(1-k q^{n-1}-k q^n+k
q^{2n})} {q^n(q,q^2;q)_n}. \notag
\end{align}
\begin{align}\label{H13'ab}
\wpalpha_n(q,k)&=0, &&\phantom{aasdfasdsasdfgasaasdfasasdfd}(\textbf{H13'})&\\
\wpbeta_n(q,k)&=\frac{ (k,k/q;q)_{n} } {(q,q^2;q)_n}. \notag
\end{align}
\begin{align}\label{H17'ab}
\wpalpha_n(1,q)&=(-1)^n (1+q^n)q^{(n^2-n)/2},&&\phantom{asdadadsadaisadsd}(\textbf{H17'})&\\
\wpbeta_n(1,q)&=\frac{ (-1)^n(k;q)_{n} q^{n(n-1)/2}k^n} {(q;q)_n}.
\notag
\end{align}
It turns out that eight of these are special cases of a more
general WP-Bailey pair. In attempting to prove that these were
indeed WP-Bailey pairs, we observed  that the following WP-Bailey
pair of Singh (\cite{S94} (see also Andrews and Berkovich
\cite{AB02}),
\begin{align}\label{wpAB}
\wpalpha_n'&=\frac{(1-a
q^{2n})(a,c,d,a^2q/kcd;q)_n}{(1-a)(q,aq/c,aq/d,
kcd/a;q)}\left( \frac{k}{a}\right)^n,\\
\wpbeta_n'&=\frac{(kc/a,kd/a,k,aq/cd;q)_n}{(aq/c,aq/d,q,kcd/a;q)_n},
\notag
\end{align}
is a lift of a standard Bailey pair of Slater \cite[Equation (4.1),
page 469]{S51}, {\allowdisplaybreaks
\begin{align}\label{wpS}
\alpha_n&=\frac{(1-a q^{2n})(a,c,d;q)_n}{(1-a)(aq/c,aq/d,
q;q)_n}\left( \frac{-a}{cd}\right)^n q^{(n^2+n)/2},\\
\beta_n&=\frac{(aq/cd;q)_n}{(aq/c,aq/d,q;q)_n}. \notag
\end{align}}
\begin{remark} (1) We have replaced the $\rho_1$ and $\rho_2$ in
\cite{AB02} with $c$ and $d$, to maintain consistency with
Slater's notation in \cite{S51}.

(2) Slater did not state the pair \eqref{wpS} explicitly as a Bailey
pair, but instead listed many special cases (see Tables \textbf{B},
\textbf{F}, \textbf{E} and \textbf{H} in \cite{S51}).
\end{remark}

By making the correct choices for $a$, $c$ and $d$, it can be shown
that the WP-Bailey pairs at \eqref{E7'ab},  \eqref{F3'ab},
\eqref{F4'ab}, \eqref{H4'ab}, \eqref{H6'ab}, \eqref{H7'ab} ,
\eqref{H13'ab} and \eqref{H17'ab} are all special cases of
\eqref{wpS}. The proofs for the remaining pairs do not follow in the
same way, because the corresponding standard Bailey pairs are
\emph{not} derived by substituting directly into \eqref{wpS}. It is
necessary to first derive some other preliminary results.

\begin{theorem}\label{wpt1} Suppose that
$(\wpalpha_n(a),\wpbeta_n(a,k))$ is a WP-Bailey pair such that
$\wpalpha_n(a)$ is independent of\, $k$. Then
$(\wpalpha_n^*(a),\wpbeta_n^*(a,k))$ is a WP-Bailey pair, where
{\allowdisplaybreaks
\begin{align*}
\wpalpha_n^*(a) &= (aq^n+q^{-n})\wpalpha_n(a)\\
\wpbeta_n^*(a,k) &=
\frac{(1+aq^{2n})\wpbeta_n(a,k)-(1-k)\left(1-\frac{k}{a}\right)\wpbeta_{n-1}(a,kq)}{q^{n}}
-a\displaystyle{\frac{(k,k/a;q)_{n}}{(aq,q;q)_{n}}}.\\
\end{align*}
}
\end{theorem}

\bP From the definition of a WP-Bailey pair,
{\allowdisplaybreaks
\begin{align*} \wpbeta_n(a,kq) &= \sum_{r=0}^n \frac{(kq;q)_{n+r}
(kq/a;q)_{n-r}}
{(aq;q)_{n+r}(q;q)_{n-r}}\wpalpha_r(a)\\
&= \sum_{r=0}^{n+1} \frac{(k;q)_{n+1+r}
(k/a;q)_{n+1-r}(1-aq^{n+1+r})(1-q^{n+1-r})}
{(1-k)(1-k/a)(aq;q)_{n+1+r}(q;q)_{n+1-r}}\wpalpha_r(a)\\
&=\frac{(1+aq^{2n+2})}{(1-k)(1-k/a)} \sum_{r=0}^{n+1}
\frac{(k;q)_{n+1+r} (k/a;q)_{n+1-r}}
{(aq;q)_{n+1+r}(q;q)_{n+1-r}}\wpalpha_r(a)\\
&-\frac{q^{n+1}}{(1-k)(1-k/a)} \sum_{r=0}^{n+1}
\frac{(k;q)_{n+1+r} (k/a;q)_{n+1-r}}
{(aq;q)_{n+1+r}(q;q)_{n+1-r}}\wpalpha_r(a)(aq^r+q^{-r})\\
&=\frac{(1+aq^{2n+2})}{(1-k)(1-k/a)}\wpbeta_{n+1}(a,k)\\
&-\frac{q^{n+1}}{(1-k)(1-k/a)} \left[a\frac{(k;q)_{n+1}
(k/a;q)_{n+1}}{(aq;q)_{n+1}(q;q)_{n+1}} +
\wpbeta_{n+1}^*(a,k)\right].
\end{align*} }
The result follows upon replacing $n$ with $n-1$.\eP

The following corollary is immediate, upon setting $k=0$.
\begin{corollary}\label{cor1}
Suppose that $(\alpha_n,\beta_n)$ is a Bailey pair relative to
$a$. Then so is $(\alpha_n^*,\beta_n^*)$, where \begin{align*}
\alpha^*_n &= (aq^{n}+q^{-n})\alpha_n\\
\beta^*_n &=\frac{(1+aq^{2n})\beta_n
 -\beta_{n-1}}{q^{n}}-\frac{a}{(aq,q;q)_{n}}.
\end{align*}
\end{corollary}
The following theorem follows directly from the definition of a
WP-Bailey pair, so the proof is omitted.

\begin{theorem}\label{wpt2}
Let $(\wpalpha_n^{(1)}(a,k), \wpbeta_n^{(1)}(a,k))$ and
$(\wpalpha_n^{(2)}(a,k), \wpbeta_n^{(2)}(a,k))$ be WP-Bailey pairs and
let $c_1$ and $c_2$ be constants. Then   $(\wpalpha_n(a,k),
\wpbeta_n(a,k))$ is a WP-Bailey pair, where
\begin{align*}
\wpalpha_n(a,k)&=c_1\wpalpha_n^{(1)}(a,k)+c_2\wpalpha_n^{(2)}(a,k),\\
\wpbeta_n(a,k)&=c_1\wpbeta_n^{(1)}(a,k)+c_2\wpbeta_n^{(2)}(a,k)
+(1-c_1-c_2)\frac{(k,k/a;q)_{n}}{(aq,q;q)_{n}}.
\end{align*}
\end{theorem}
Note for later the special case $c_2=0$. The following corollary is
immediate upon setting $k=0$.
\begin{corollary}\label{cor2}
Let $(\alpha_n^{(1)}, \beta_n^{(1)} )$ and $(\alpha_n^{(2)},
\beta_n^{(2)} )$ be Bailey pairs relative to $a$, and let $c_1$ and
$c_2$ be constants. Then   $(\alpha_n, \beta_n)$ is a Bailey pair
relative to $a$, where
\begin{align*}
\alpha_n&=c_1\alpha_n^{(1)}+c_2\alpha_n^{(2)},\\
\beta_n&=c_1\beta_n^{(1)}+c_2\beta_n^{(2)}
+\frac{(1-c_1-c_2)}{(aq,q;q)_{n}}.
\end{align*}
\end{corollary}
\begin{remark}
 Some the Bailey pairs derived by Slater \cite{S51,S52}
follow from other pairs derived by her in \cite{S51,S52} as a result
of Corollaries \ref{cor1} and \ref{cor2}.
\end{remark}

We are now in a position to prove that the pairs at \eqref{H3'ab},
\eqref{H5'ab}, \eqref{H8'ab} and \eqref{H12'ab} are indeed WP-Bailey
pairs.

\begin{corollary}
The pair of sequences $(\wpalpha_n(1,k), \wpbeta_n(1,k))$, where
\begin{align*}
\wpalpha_n(1,k)&=(-1)^n q^{-n^2/2}\left(q^{-3 n/2}+q^{3 n/2}\right), &&\phantom{asdadasdadaasd}(\textbf{H3'})&\\
\wpbeta_n(1,k)&=\frac{(-1)^n (1-kq^n+kq^{2n})
(k;q)_{n}}{q^{(n^2+3n)/2}(q;q)_n}, \notag
\end{align*}
is a WP-Bailey pair.
\end{corollary}
\begin{proof}From what has been said previously, the pair at
\eqref{H4'ab}, namely
\begin{align*}
\wpalpha_n^{(1)}(1,k)&=(-1)^n q^{-n^2/2}\left(q^{-n/2}+q^{n/2}\right),&&\phantom{asdadadsaasadasd}(\textbf{H4'})&\\
\wpbeta_n^{(1)}(1,k)&=\frac{(-1)^n (k;q)_{n}}{q^{(n^2+n)/2}(q;q)_n}.
\notag
\end{align*}
is a WP-Bailey pair. From Theorem \ref{wpt1}, with $a=1$,
\begin{align*}
\wpalpha_n^{(2)}(1,k)&= (q^n+q^{-n})\wpalpha_n^{(1)}(1,k)\\
&=
(-1)^n q^{-n^2/2}\left(q^{-3n/2}+q^{3n/2}+q^{-n/2}+q^{n/2}\right),\\
\wpbeta_n^{(2)}(1,k)&=\frac{(1+q^{2n})\wpbeta_n^{(1)}(1,k)-(1-k)
\left(1-k\right)\wpbeta_{n-1}^{(1)}(1,kq)}{q^{n}}
-\displaystyle{\frac{(k,k;q)_{n}}{(q,q;q)_{n}}}\\
&=\frac{(-1)^n (k;q)_{n}}{q^{(n^2+n)/2}(q;q)_n}\frac{1+q^n-k q^n+k
q^{2n}}{q^n}-\displaystyle{\frac{(k,k;q)_{n}}{(q,q;q)_{n}}}, \notag
\end{align*}
is a WP-Bailey pair. The result now follows from Theorem \ref{wpt2},
upon setting $a=1$, $c_1=-1$, $c_2=1$ and letting
$(\wpalpha_n^{(1)}(1,k), \wpbeta_n^{(1)}(1,k))$ and
$(\wpalpha_n^{(2)}(1,k)$, $\wpbeta_n^{(2)}(1,k))$ be as stated above.
\end{proof}

We next prove that the pair at \eqref{H8'ab} is a WP-Bailey pair.
\begin{corollary}
The pair of sequences $(\wpalpha_n(1,k), \wpbeta_n(1,k))$, where
\begin{align*}
\wpalpha_n(1,k)&=(-1)^n(q^{-n}+q^{n}), &&\phantom{asdadadsadaasdadasd}(\textbf{H8'})&\\
\wpbeta_n(1,k)&=(-1)^n\frac{ (k^2;q^2)_{n} (1+kq^{2n})}
{(1+k)q^n(q^2;q^2)_n}. \notag
\end{align*}
is a WP-Bailey pair.
\end{corollary}
\begin{proof}
The proof is similar to that for \eqref{H3'ab} in the corollary
above, except we start with
\begin{align*}
\wpalpha_n^{(1)}(1,k)&=2(-1)^n, &&\phantom{asdadadsadaasdadssdsasadasd}(\textbf{H7'})&\\
\wpbeta_n^{(1)}(1,k)&=(-1)^n\,\frac{ (k^2;q^2)_{n}}{(q^2;q^2)_n}.
\notag
\end{align*}
Theorem \ref{wpt1} gives
\begin{align*}
\wpalpha_n^{(2)}(1,k)&=2(-1)^n(q^n+q^{-n}), \\
\wpbeta_n^{(2)}(1,k)&=(-1)^n\,\frac{ (k^2;q^2)_{n}(2+2
q^{2n})}{q^n(1+k)(q^2;q^2)_n}-\displaystyle{\frac{(k,k;q)_{n}}{(q,q;q)_{n}}},
\notag
\end{align*}
is a WP-Bailey pair. The result follows once again from Theorem
\ref{wpt2}, upon setting $c_2=1/2$ and $c_1=0$.
\end{proof}
We next give  proofs for the two remaining two pairs at
\eqref{H5'ab} and \eqref{H12'ab}.
\begin{corollary}
The pairs of sequences
\begin{align*}
\wpalpha_n(1,k)&=q^{-n}-q^{n}, &&\phantom{asdadadsadaadasd}(\textbf{H5'})&\\
\wpbeta_n(1,k)&=\frac{ (k,k;q)_{n}(1-2 k q^n + k q^{2 n})}
{(1-k)q^{n}(q,q;q)_n}. \notag
\end{align*}
and
\begin{align*}
\wpalpha_n(q,k)&=q^{-n}-q^{n+1}, &&\phantom{aad}(\textbf{H12'})&\\
\wpbeta_n(q,k)&=\frac{ (k;q)_{n} (k;q)_{n-1}(1-k q^{n-1}-k q^n+k
q^{2n})} {q^n(q,q^2;q)_n}. \notag
\end{align*}
are WP-Bailey pairs.
\end{corollary}
\begin{proof}
Let $c \to \infty$ (or $c \to 0$) and set $d=q$  in \eqref{wpAB} to
get that
\begin{align*}
\alpha_n(a,k) &= \frac{q^{-n}-aq^{n}}{1-a},\\
\beta_{n}(a,k) &= \frac{(kq/a,k;q)_n}{(a,q;q)_n q^n}
\end{align*}
is a WP-Bailey pair. Now apply Theorem \ref{wpt2} with $c_1=1-a$ and
$c_2=0$ (with $(\wpalpha_n^{(1)}(a,k), \wpbeta_n^{(1)}(a,k))$ $=$
$(\wpalpha_n^{(2)}(a,k), \wpbeta_n^{(2)}(a,k))$ $=$ $(\wpalpha_n(a,k),
\wpbeta_n(a,k))$) to get that
\begin{align*}
\wpalpha_n^*(a,k) &= q^{-n}-aq^{n},\\
\wpbeta_{n}^*(a,k) &= \frac{(kq/a,k;q)_n}{(aq;q)_{n-1}(q;q)_n
q^n}+a\frac{(k,k/a;q)_n}{(aq,q;q)_n },
\end{align*}
is a WP-Bailey pair. The pairs in the statement of the corollary
are, respectively, the cases $a=1$ and $a=q$.
\end{proof}

\begin{remark} If lifts exist for the remaining Bailey pairs found by
Slater \cite{S51, S52}, finding them will likely prove more
difficult, as experiment seems to indicate that the $\alpha_n$ are
\emph{not} independent of $k$.
\end{remark}

For completeness sake, we include the following theorem in this
section, as it gives yet another way of deriving new Bailey pairs
from existing Bailey pairs. We first note the identities
\begin{equation}\label{id1}
1=\frac{1-a q^{n+r+1}}{1-a q^{2r+1}}-a q^{2r+1}\frac{1-q^{n-r}}{1-a
q^{2r+1}}.
\end{equation}
\begin{equation}\label{id2}
1=q^{-n+r}\frac{1-a q^{n+r+1}}{1-a
q^{2r+1}}-q^{-n+r}\frac{1-q^{n-r}}{1-a q^{2r+1}}.
\end{equation}

\begin{theorem}\label{bptother}
Suppose $(\alpha_n(a,q), \beta_n(a,q))$ is a Bailey pair with
respect to $a$. Then so are the pairs $(\alpha_n^*(a,q),
\beta_n^*(a,q))$ and $(\alpha_n^{\dagger}(a,q),
\beta_n^{\dagger}(a,q))$, where $\alpha_0^*(a,q)=
\beta_0^*(a,q)=\alpha_0^{\dagger}(a,q)= \beta_0^{\dagger}(a,q)=1$,
and for $n>0$,{\allowdisplaybreaks
\begin{align*}
\alpha_n^*(a,q)&=(1-aq)\left ( \frac{\alpha_n(aq,q)}{1-a q^{2n+1}}-a
q^{2n-1}\frac{\alpha_{n-1}(aq,q)}{1-a q^{2n-1}}\right),\\
\beta_n^*(a,q)) &=\beta_n(aq,q),
\end{align*}
\begin{align*}
\alpha_n^{\dagger}(a,q)&=(1-aq)\left ( q^n \frac{
\alpha_n(aq,q)}{1-a q^{2n+1}}-
q^{n-1}\frac{\alpha_{n-1}(aq,q)}{1-a q^{2n-1}}\right),\\
\beta_n^{\dagger}(a,q)) &=q^n \beta_n(aq,q).
\end{align*}
}
\end{theorem}
\begin{proof}
From the definition of a Bailey pair, {\allowdisplaybreaks
\begin{align*}
\beta_n(aq,q)&=\sum_{r=0}^{n}
\frac{\alpha_r(aq,q)}{(q;q)_{n-r}(aq^2;q)_{n+r}}\\
&=\sum_{r=0}^{n-1}
\frac{\alpha_r(aq,q)}{(q;q)_{n-r}(aq^2;q)_{n+r}}+\frac{\alpha_n(aq,q)}{(aq^2;q)_{2n}}\\
&=(1-aq)\sum_{r=0}^{n-1}
\frac{\alpha_r(aq,q)}{(q;q)_{n-r}(aq;q)_{n+r+1}}+\frac{\alpha_n(aq,q)}{(aq^2;q)_{2n}}\\
&=(1-aq)\sum_{r=0}^{n-1} \frac{\alpha_r(aq,q)}{(1-a
q^{2r+1})(q;q)_{n-r}(aq;q)_{n+r}}\\
&\phantom{asd}-(1-aq)\sum_{r=0}^{n-1} \frac{a
q^{2r+1}\alpha_r(aq,q)}{(1-a
q^{2r+1})(q;q)_{n-r-1}(aq;q)_{n+r+1}}+\frac{\alpha_n(aq,q)}{(aq^2;q)_{2n}}\\
&=(1-aq)\sum_{r=0}^{n-1} \frac{\alpha_r(aq,q)}{(1-a
q^{2r+1})(q;q)_{n-r}(aq;q)_{n+r}}\\
&\phantom{asd}-(1-aq)\sum_{r=1}^{n} \frac{a
q^{2r-1}\alpha_{r-1}(aq,q)}{(1-a
q^{2r-1})(q;q)_{n-r}(aq;q)_{n+r}}+\frac{\alpha_n(aq,q)}{(aq^2;q)_{2n}}.
\end{align*}
} The next-to-last equality follows from \eqref{id1} and the result
now follows for $(\alpha_n^*(a,q), \beta_n^*(a,q))$. The result for
$(\alpha_n^{\dagger}(a,q), \beta_n^{\dagger}(a,q))$ follows
similarly, except we use \eqref{id2} at the next to last step.
\end{proof}

\begin{corollary}
The pairs $(\alpha_n^*(a,q), \beta_n^*(a,q))$ and
$(\alpha_n^{\dagger}(a,q), \beta_n^{\dagger}(a,q))$ are Bailey pairs
with respect to $a$, where $\alpha_0^*(a,q)=
\beta_0^*(a,q)=\alpha_0^{\dagger}(a,q)= \beta_0^{\dagger}(a,q)=1$,
and for $n>0$,
\begin{align*}
\alpha_n^*(a,q)&=\frac{(aq,c,d;q)_n}{(aq^2/c,aq^2/d, q;q)_n}\left(
\frac{-a q}{cd}\right)^n q^{(n^2+n)/2} \\
&\phantom{asd}- a q^{2n-1}\frac{(aq,c,d;q)_{n-1}}{(aq^2/c,aq^2/d,
q;q)_{n-1}}\left( \frac{-a q}{cd}\right)^{n-1} q^{(n^2-n)/2},\\
\beta_n^*(a,q)) &=\frac{(aq^2/cd;q)_n}{(aq^2/c,aq^2/d,q;q)_n},
\end{align*}
\begin{align*}
\alpha_n^{\dagger}(a,q)&=q^n\frac{(aq,c,d;q)_n}{(aq^2/c,aq^2/d,
q;q)_n}\left( \frac{-a q}{cd}\right)^n q^{(n^2+n)/2}\\
&\phantom{asd} - q^{n-1}\frac{(aq,c,d;q)_{n-1}}{(aq^2/c,aq^2/d,
q;q)_{n-1}}\left( \frac{-a q}{cd}\right)^{n-1} q^{(n^2-n)/2},\\
\beta_n^{\dagger}(a,q)) &=q^n
\frac{(aq^2/cd;q)_n}{(aq^2/c,aq^2/d,q;q)_n}.
\end{align*}
\end{corollary}
\begin{proof}
This follows directly from applying Theorem \ref{bptother} to the
Bailey pair at \eqref{wpS}.
\end{proof}

\section{ Dual WP-Bailey pairs}

We next consider a natural pairing of  WP-Bailey pairs. We first
recall that
\begin{equation}\label{bin1/qeq}
(a;1/q)_n = \frac{(-a)^n(1/a;q)_n}{q^{n(n-1)/2}}.
\end{equation}
Andrews showed in \cite{A84}  that if  $(\alpha_{n}(a,q),
\beta_{n}(a,q))$ is a Bailey pair relative to $a$, then
$(\alpha_{n}^*(a,q), \beta_{n}^*(a,q))$ is also a Bailey pair
relative to $a$, where
\begin{align*}
\alpha_{n}^*(a,q) &= a^n q^{n^2}\alpha_{n}(1/a,1/q),\\
 \beta_{n}^*(a,q) &= a^{-n}q^{-n^2-n}\beta_n(1/a,1/q).
\end{align*}

The pair $(\alpha_{n}^*(a,q), \beta_{n}^*(a,q))$ is called \emph{the
dual} of $(\alpha_{n}(a,q), \beta_{n}(a,q))$. Note that the dual of
$(\alpha_{n}^*(a,q), \beta_{n}^*(a,q))$ is $(\alpha_{n}(a,q),$
$\beta_{n}(a,q))$. As an example, the dual of the Bailey pair
{\allowdisplaybreaks
\begin{align*}
\alpha_n(1,q)&=(-1)^n q^{3n^2/2}\left(q^{-3 n/2}+q^{3 n/2}\right),
&&\phantom{asdadasdadaasd}(\textbf{B2})&\\
\beta_n(1,q)&=\frac{q^n }{(q;q)_n}. \notag
\end{align*}}
is the Bailey pair
\begin{align*}
\alpha_n(1,q)&=(-1)^n q^{-n^2/2}\left(q^{-3 n/2}+q^{3 n/2}\right), &&\phantom{asdadasdadaasd}(\textbf{H3})&\\
\beta_n(1,q)&=\frac{(-1)^n }{q^{(n^2+3n)/2}(q;q)_n}. \notag
\end{align*}

This concept of duality can be extended to WP-Bailey pairs.
\begin{theorem}\label{WPBdual}
Suppose $(\wpalpha_{n}(a,k,q), \wpbeta_{n}(a,k,q))$ is a WP-Bailey pair.
Then $(\wpalpha_{n}^*(a,k,q), \wpbeta_{n}^*(a,k,q))$ is also a WP-Bailey
pair, where
\begin{align*}
\wpalpha_{n}^*(a,k,q) &= \wpalpha_{n}(1/a,1/k,1/q),\\
 \wpbeta_{n}^*(a,k,q) &= \left(\frac{k}{aq}\right)^{2n}\wpbeta_n(1/a,1/k,1/q).
\end{align*}
\end{theorem}
\begin{proof}
Replace $a$ by $1/a$, $k$ by $1/k$ and $q$ by $1/q$ in
\eqref{WPpair} and use \eqref{bin1/qeq} to simplify the resulting
expression.
\end{proof}
As with standard Bailey pairs, we refer to the pair
$(\wpalpha_{n}^*(a,k,q), \wpbeta_{n}^*(a,k,q))$ in Theorem \ref{WPBdual}
as \emph{the dual} of $(\wpalpha_{n}(a,k,q), \wpbeta_{n}(a,k,q))$. Note
that, as above, the dual of $(\wpalpha_{n}^*(a,k,q),
\wpbeta_{n}^*(a,k,q))$ is $(\wpalpha_{n}(a,k,q),$ $\wpbeta_{n}(a,k,q))$.

We also remark that it is possible to use these duality
constructions to derive new Bailey- or WP-Bailey pairs.

\begin{corollary}
The pair of sequences $(\wpalpha_{n}^*(1,k), \wpbeta_{n}^*(1,k))$ is a
WP-Bailey pair, where
\begin{align}\label{H3'dual}
\wpalpha_n^*(1,k)&=(-1)^n q^{n^2/2}\left(q^{-3 n/2}+q^{3 n/2}\right), \\
\wpbeta_n^*(1,k)&=\frac{(-1)^n (1-q^n+kq^{2n})
(k;q)_{n}k^{n-1}q^{(n^2-3n)/2}}{(q;q)_n}. \notag
\end{align}
\end{corollary}
\begin{proof}
The pair at \eqref{H3'dual} is the dual of the pair at
\eqref{H3'ab}:
\begin{align*}
\wpalpha_n(1,k)&=(-1)^n q^{-n^2/2}\left(q^{-3 n/2}+q^{3 n/2}\right), &&\phantom{asdadasdadaasd}(\textbf{H3'})&\\
\wpbeta_n(1,k)&=\frac{(-1)^n (1-kq^n+kq^{2n})
(k;q)_{n}}{q^{(n^2+3n)/2}(q;q)_n}. \notag
\end{align*}
\end{proof}
\begin{remark} The Bailey pair derived from \eqref{H3'dual} by setting
$k=0$ does not appear in Slater's lists of Bailey pairs in
\cite{S51, S52}.
\end{remark}

\section{Basic Hypergeometric series identities and identities of
 Rogers - Ramanujan-Slater type}

Each of the WP-Bailey pairs found above may be substituted into
\eqref{fbt1eq} and \eqref{wpeq},  leading to possibly new identities
between basic hypergeometric series and/or new identities of the
Rogers-Ramanujan type. We illustrate this by considering the
WP-Bailey pair from \eqref{H3'dual}. We believe that the following
identities are new.
\begin{corollary}\label{cor3}
Let $N$ be a positive integer and suppose $k \not = 0$.
\begin{multline}\label{cor3eq1}
\hspace{-4pt}\sum_{n=0}^{N}\frac{\left(q\sqrt{k},-q\sqrt{k},k,y,z,\frac{k
q^{N+1}}{yz},
\frac{1-\sqrt{1-4k}}{2}q,\frac{1+\sqrt{1-4k}}{2}q,q^{-N};q\right)_n}
{\left(\sqrt{k},-\sqrt{k},\frac{k q}{y},\frac{k q}{z},y z q^{-N}, k
q^{1+N},\frac{1-\sqrt{1-4k}}{2},
\frac{1+\sqrt{1-4k}}{2},q;q\right)_n}(-k)^n
q^{\frac{n^2-n}{2}}\\
=\frac{\left(q k, \frac{q k}{y z}, \frac{q }{y}, \frac{q
}{z};q\right)_N}{\left(\frac{q k}{y}, \frac{q k}{z}, q , \frac{q
}{y z};q\right)_N}\phantom{asdasdasasdasdasdadadasdasdasasdddaadasd}\\
\times \left( 1+\sum_{n=1}^{N} \frac{ \left(y,z,\frac{k  q^{N+1}}{y
z},q^{-N};q\right)_n} {\left(\frac{q}{y},\frac{q}{z},\frac{y z
q^{-N}}{k},  q^{1+N};q\right)_n} \left(\frac{-1}{k}\right)^n
q^{(n^2-n)/2}\left(1+q^{3n}\right) \right).
\end{multline}
\begin{multline}\label{cor3eq2}
\hspace{-4pt}\sum_{n=0}^{\infty}\frac{\left(q\sqrt{k},-q\sqrt{k},k,
\frac{1-\sqrt{1-4k}}{2}q,\frac{1+\sqrt{1-4k}}{2}q;q\right)_n
(q;q^2)_n k^n q^{(n^2-3n)/2}}
{\left(\sqrt{k},-\sqrt{k},\frac{1-\sqrt{1-4k}}{2},
\frac{1+\sqrt{1-4k}}{2},q;q\right)_n (k^2q;q^2)_n}\\
=\frac{1+k}{q}\frac{\left(q^2 k^2, ;q^2\right)_{\infty}}{\left( q
k^2 ;q^2\right)_{\infty}}.
\end{multline}
\begin{equation}\label{cor3eq3}
\hspace{-4pt}\sum_{n=0}^{\infty}
\frac{\left(q\sqrt{k},-q\sqrt{k},k,
\frac{1-\sqrt{1-4k}}{2}q,\frac{1+\sqrt{1-4k}}{2}q;q\right)_n
 (-k)^n q^{3(n^2-n)/2}}
{\left(\sqrt{k},-\sqrt{k},\frac{1-\sqrt{1-4k}}{2},
\frac{1+\sqrt{1-4k}}{2},q;q\right)_n
}
=0.
\end{equation}
\begin{multline}\label{cor3eq4}
\hspace{-4pt}\sum_{n=0}^{\infty}
\frac{\left(q\sqrt{k},-q\sqrt{k},k,\sqrt{q},
\frac{1-\sqrt{1-4k}}{2}q,\frac{1+\sqrt{1-4k}}{2}q;q\right)_n
 k^n q^{n^2-3n/2}}
{\left(\sqrt{k},-\sqrt{k},k\sqrt{q},\frac{1-\sqrt{1-4k}}{2},
\frac{1+\sqrt{1-4k}}{2},q;q\right)_n \phantom{asdasaasdsdsdd}}\\ =
q^{-1/2}\frac{(k q;q)_{\infty}}{(k\sqrt{q};q)_{\infty}}.
\end{multline}
\end{corollary}
\begin{proof}
The identity at \eqref{cor3eq1} follows  upon substituting the pair
from \eqref{H3'dual} into \eqref{fbt1eq}, setting $a=1$, and using
the fact that
\[
1-q^n+kq^{2n} =k \frac{\left(
\frac{1-\sqrt{1-4k}}{2}q,\frac{1+\sqrt{1-4k}}{2}q;q\right)_n}
{\left(\frac{1-\sqrt{1-4k}}{2}, \frac{1+\sqrt{1-4k}}{2};q\right)_n}.
\]

Next, let $N \to \infty$ in \eqref{cor3eq1} to get that
\begin{multline}\label{cor3eq1a}
\hspace{-4pt}\sum_{n=0}^{\infty}\frac{\left(q\sqrt{k},-q\sqrt{k},k,y,z,
\frac{1-\sqrt{1-4k}}{2}q,\frac{1+\sqrt{1-4k}}{2}q;q\right)_n}
{\left(\sqrt{k},-\sqrt{k},\frac{k q}{y},\frac{k
q}{z},\frac{1-\sqrt{1-4k}}{2},
\frac{1+\sqrt{1-4k}}{2},q;q\right)_n}\left(\frac{-k}{y z} \right)^n
q^{\frac{n^2-n}{2}}\\
=\frac{\left(q k, \frac{q k}{y z}, \frac{q }{y}, \frac{q
}{z};q\right)_{\infty}}{\left(\frac{q k}{y}, \frac{q k}{z}, q ,
\frac{q }{y z};q\right)_{\infty}}\left( 1+\sum_{n=1}^{\infty} \frac{
\left(y,z;q\right)_n} {\left(\frac{q}{y},\frac{q}{z};q\right)_n}
\left(\frac{-1}{ y z}\right)^n
q^{(n^2-n)/2}\left(1+q^{3n}\right)\right).
\end{multline}

The identities at \eqref{cor3eq2}, \eqref{cor3eq3} and
\eqref{cor3eq4} follow, respectively, from letting $(y,z)\to
(\sqrt{q}, -\sqrt{q})$, $(y,z)\to (\infty, \infty)$ and $(y,z)\to
(\sqrt{q}, \infty)$, and using the Jacobi triple product identity
\eqref{JTP} to sum the resulting series on the right sides.
\end{proof}

\section{Double-sum identities of the Rogers-Ramanujan-Slater
type}

If $( \alpha_n )_{n\geq 0}$ is any sequence with $\alpha_0=1$, then trivially
\[
\left(\alpha_n, \sum_{j=0}^{n}
\frac{\left(k/a;q\right)_{n-j}(k;q)_{n+j}}{(q;q)_{n-j}(aq;q)_{n+j}}\,\alpha_{j}
\right )
\]
is a WP-Bailey pair. If this pair is substituted into
\eqref{wpeq},  then after switching the order of summation on the
left side and re-indexing, we get the following theorem.

\begin{theorem}\label{6t4}
Let $(\alpha_n)_{n\geq 0}$ be any sequence with $\alpha_0=1$.
Then, subject to suitable convergence conditions,
{\allowdisplaybreaks
\begin{multline}\label{wpeq4}
\sum_{j,n=0}^{\infty} \frac{(1-k
q^{2n+2j})(k/a;q)_{n}(k;q)_{n+2j}(y,z;q)_{n+j}}
{(1-k)(q;q)_n(a q;q)_{n+2j}(q k/y,q k/z;q)_{n+j}}\left( \frac{q a}{y z }\right )^{n+j} \alpha_j\\
 =\frac{(q k,q k/yz,q a/y,q a/z;q)_{\infty}} {(q k/y,q k/z,q a,q
a/yz;q)_{\infty}} \sum_{n=0}^{\infty}\frac{(y,z;q)_{n}}{(q a/y ,q
a/z;q)_n}\left (\frac{q a}{y z}\right)^{n}\alpha_n.
\end{multline}
}
\end{theorem}
This last equation can be used to derive, almost trivially, a large
number of double-sum series = product identities. Firstly, any
identity on the Slater list can be extended to a double-sum identity
involving the free parameter $k$, and which reverts back to the
original single-sum identity upon setting $k=0$. This is simply done
by inserting the sequence $\alpha_n$ from the same Bailey pair
$(\alpha_n, \beta_n)$, and making the same choices for $y$ and $z$,
as Slater did to derive the original identity. Secondly, we can
choose $(\alpha_n)$ so that the series on the right of \eqref{wpeq4}
becomes one of the series in an identity on the Slater list, so that
the right side once again can be expressed as an infinite product.
Thirdly, we can choose $\alpha_n$ so that the series on the right
side becomes summable via the Jacobi triple product identity or the
quintuple product identity. We illustrate each of these methods of
generating a double-sum series = product by giving an example in
each case (the resulting identities are new).

We first consider the standard pair \textbf{B1} from Slater's
paper \cite{S51}. This pair has
\begin{equation}\label{B1pair}
\alpha_j=
\begin{cases}1,&j=0,\\
q^{3j^2/2-j/2}(-1)^j(1+q^j), &j>0,
\end{cases}
\end{equation}
and leads to the first Rogers-Ramanujan identity.
\begin{equation}\label{RR1}
\sum_{n=0}^{\infty} \frac{q^{n^2}}{(q;q)_{n}} =
\frac{1}{(q,q^4;q^{5})_{\infty}}.
\end{equation}
\begin{corollary}\label{crr1}
For $k \not = 1$,
\begin{multline}\label{B1/2pairpr3}
\sum_{j,n =0}^{\infty}\frac{ (1-k
q^{2n+2j})(1+q^j)(k;q)_{n}(k;q)_{n+2j}q^{(5j^2-j)/2+2nj+n^2}(-1)^j
}{ (1-k)(q;q)_{n}(q;q)_{n+2j}}\\= \frac{(k
q;q)_{\infty}}{(q,q^4;q^5)_{\infty}} +\frac{(k
q;q)_{\infty}}{(q;q)_{\infty}}.
\end{multline}
\end{corollary}
\begin{proof}
First, set $a=1$ and let $y, z \to \infty$ in \eqref{wpeq4} to get
{\allowdisplaybreaks
\begin{equation}\label{wpeq4B1}
\sum_{j,n=0}^{\infty} \frac{(1-k q^{2n+2j})(k;q)_{n}(k;q)_{n+2j}}
{(1-k)(q;q)_n( q;q)_{n+2j}}q^{(n+j)^2} \alpha_j
 =\frac{(q k;q)_{\infty}} {(q ;q)_{\infty}} \sum_{n=0}^{\infty}
 q^{n^2}\alpha_n.
\end{equation}
}

Next, substitute for $(\alpha_n)$, writing $\alpha_0=1$ as
\[
\alpha_0=1=q^{3(0^2)/2-0/2}(-1)^0(1+q^0)-1,
\]
and use \eqref{JTP} to sum the right side, to get
\begin{multline*}
\sum_{j,n =0}^{\infty}\frac{ (1-k
q^{2n+2j})(1+q^j)(k;q)_{n}(k;q)_{n+2j}q^{(5j^2-j)/2+2nj+n^2}(-1)^j
}{ (1-k)(q;q)_{n}(q;q)_{n+2j}}\\
-\sum_{n =0}^{\infty}\frac{ (1-k q^{2n})(k,k;q)_{n}q^{n^2} }{
(1-k)(q,q;q)_{n}}=\frac{(q k;q)_{\infty}} {(q
;q)_{\infty}}(q^2,q^3,q^5;q^5)_{\infty} =\frac{(k
q;q)_{\infty}}{(q,q^4;q^5)_{\infty}}.
\end{multline*}
The fact that
\[
\sum_{n =0}^{\infty}\frac{ (1-k q^{2n})(k,k;q)_{n}q^{n^2} }{
(1-k)(q,q;q)_{n}} =\frac{(k q;q)_{\infty}}{(q;q)_{\infty}}
\]
follows as a special case of the following identity, which is a
special case of an identity due to Jackson \cite{J21} (set $a=k$,
$b=k$ and let $c, d \to \infty$).
\begin{align}\label{6phi5eq}
&\PHI{6}{5}\phiargs{a,q\sqrt{a},-q\sqrt{a},b,c,d}{
\sqrt{a},-\sqrt{a},aq/b,aq/c,aq/d}{q}{\frac{aq}{bcd}}=
\frac{(aq,aq/bc,aq/bd,aq/cd;q)_\infty}{(aq/b,aq/c,aq/d,aq/bcd;q)_\infty}.
\end{align}
\end{proof}

\begin{corollary}
{\allowdisplaybreaks
\begin{multline}\label{wpeq4B1'}
\sum_{j,n=0}^{\infty} \frac{(1-k
q^{2n+2j})\left(k/a;q\right)_{n}(k;q)_{n+2j}(y q^j,z
q^j;q)_{n}\left(q a/y,q a/z;q\right)_{j}} {(1-k)(q;q)_n(a
q;q)_{n+2j} \left(q k/y,q k/z;q\right)_{n+j} (q;q)_j}\\\times \left(
\frac{q a}{y z }\right )^{n} q^{j^2}
 =\frac{(q k,q k/yz,q a/y,q a/z;q)_{\infty}} {(q k/y,q k/z,q a,q
a/yz;q)_{\infty}} \frac{1}{(q,q^4;q^5)_{\infty}}.
\end{multline}
}
\end{corollary}
\begin{proof}
Set
\[
\alpha_n =\frac{(q a/y ,q a/z;q)_n}{(y,z;q)_{n}}\left (\frac{y z}{q
a}\right)^{n} \frac{q^{n^2}}{(q;q)_n}
\]
in \eqref{wpeq4}, so that the series on the left side of
\eqref{wpeq4} becomes the series on the right side of \eqref{RR1}.
\end{proof}
Corollary \ref{crr1} is an extension of the first Rogers-Ramanujan
identity, since setting $k=0$ recovers this identity, after some
series manipulations. It is possible to generalize the identity in
Corollary \ref{crr1} as follows (Corollary \ref{crr1} is the case
$s=5/2$, $r=1/2$ of the following corollary).
\begin{corollary}\label{crs}
Let $s$ be a positive rational number and $r$ a rational number. For
$k \not = 1$,
\begin{multline}\label{rseq}
\sum_{j,n =0}^{\infty}\frac{ (1-k
q^{2n+2j})(1+q^{2rj})(k;q)_{n}(k;q)_{n+2j}q^{sj^2-jr+2nj+n^2}(-1)^j
}{ (1-k)(q;q)_{n}(q;q)_{n+2j}}\\= \frac{(k
q;q)_{\infty}(q^{s-r},q^{s+r},q^{2s};q^{2s})_{\infty}}{(q;q)_{\infty}}
+\frac{(k q;q)_{\infty}}{(q;q)_{\infty}}.
\end{multline}
\end{corollary}
\begin{proof}
The proof is similar to the proof of Corollary \ref{crr1}, except we
make the substitution
\[
\alpha_n=q^{(s-1)n^2-nr}(-1)^n(1+q^{2nr})
\]
in \eqref{wpeq4B1}.
\end{proof}

\section{WP Generalizations of the Multiparameter Bailey Pairs}
In~\cite{S07}, the second author showed that more than half of
the identities in
Slater's list could be recovered by specializing
parameters in just three general Bailey pairs together with
some $q$-difference equations.

The \emph{standard multiparameter Bailey pair (SMBP)}~\cite{S07} is
defined as follows:

Let
\begin{multline} \label{SMBPdef}
\alpha_{n}^{(d,e,h)}(a,q):=
   \frac{(-1)^{n/d} a^{(h/d-1)n/e} q^{(h/d-1+1/2d)n^2/e - n/2e} (1-a^{1/e} q^{2n/e})
   }{(1-a^{1/e}) (q^{d/e};q^{d/e})_{n/d}} \\
            \times  (a^{1/e};q^{d/e})_{n/d}\ \chi(d\mid n) ,
 \end{multline} where
\[ \chi(P(n,d))= \left\{
  \begin{array}{ll}
  1 &\mbox{if $P(n,d)$ is true,}\\
  0 &\mbox{if $P(n,d)$ is false,}
  \end{array} \right.
\]
 and let $\beta_{n}^{(d,e,h)} (a,q)$ be determined by \eqref{bpeq}.

The \emph{Euler multiparameter Bailey pair (EMBP)} is given by
\begin{equation} \label{EMBPdef}
\widetilde{\alpha}_{n}^{(d,e,h)}(a,q):=q^{n(d-n)/2de} a^{-n/de}
\frac{(-a^{1/e};q^{d/e})_{n/d}}
{(-q^{d/e};q^{d/e})_{n/d}}  \alpha_{n}^{(d,e,h)}(a,q)
 \end{equation} with $\widetilde{\beta}_{n}^{(d,e,h)} (a,q)$
determined by \eqref{bpeq},
 and the \emph{Jackson-Slater multiparameter Bailey pair (JSMBP)} by
\begin{equation} \label{JSMBPdef}
\bar{\alpha}_{n}^{(d,e,k)}(a,q):= (-1)^{n/d} q^{-n^2/2de}
\frac{(q^{d/2e};q^{d/e})_{n/d}}
{(a^{1/e} q^{d/2e}; q^{d/e})_{n/d}} \alpha_{n}^{(d,e,h)}(a,q)
 \end{equation} with $\bar{\beta}_{n}^{(d,e,h)} (a,q)$ determined by
\eqref{bpeq}.

Clearly each of the $\alpha$'s in~\eqref{SMBPdef}--\eqref{JSMBPdef}
could be inserted into~\eqref{WPpair} instead of~\eqref{bpeq} to
produce WP generalizations of the multiparameter Bailey pairs.

Let us therefore define
\begin{align}
   \wpalpha_{n}^{(d,e,h)}(a,k,q) &:= \alpha_{n}^{(d,e,h)}(a,q),  \label{SWPalpha}\\
   \widetilde{\wpalpha}_{n}^{(d,e,h)}(a,k, q) &:=\widetilde{\alpha}_{n}^{(d,e,h)}(a,q), \label{EWPalpha}\\
   \bar{\wpalpha}_{n}^{(d,e,k)}(a,k, q) & :=\bar{\alpha}_{n}^{(d,e,k)}(a,q), \label{JSWPalpha}
\end{align}
and employ~\eqref{WPpair} to obtain the following corresponding
$\wpbeta$'s:
\begin{multline}
\wpbeta_n^{(d,e,h)}(a^e, k, q^e) = \frac{(k, k/a^e; q^e)_n}{(q^e, a^e q^e;q^e)_n} \\
 \times
   \sum_{r=0}^{\lfloor n/d \rfloor} \frac{(a;q^d)_r (1-aq^{2dr})
    (q^{-en}, k q^{en}; q^e)_{dr} }
   {(q^d;q^d)_r (1-a) (a^e q^{e(n+1)}, k^{-1}a^e q^{e(1-n)}; q^e)_{dr} }\\
   \times (-1)^{r} a^{(h-d+ed)r} k^{-dr} q^{(2h-2d+1)dr^2/2 +
   (2e -1)dr/2} \label{SWPbeta}
\end{multline}
\begin{multline}
\widetilde{\wpbeta}_n^{(d,e,h)}(a^e,k, q^e)
 = \frac{(k , k/a^e; q^e)_n }{(q^e, a^e q^e;q^e)_n} \\
 \times
   \sum_{r=0}^{\lfloor n/d \rfloor} \frac{(a^2;q^{2d})_r (1-aq^{2dr})
   (q^{-en}, k q^{en}; q^e)_{dr} }
   {(q^{2d};q^{2d})_r (1-a) (a^e q^{e(n+1)}, k^{-1} a^e q^{e(1-n)}; q^e)_{dr} }\\
   \times (-1)^{r} a^{(h-d-1+ed)r} k^{-dr} q^ {(h - d )dr^2 +
   edr} \label{EWPbeta}
\end{multline}
\begin{multline}
\bar{\wpbeta}_n^{(d,e,h)}(a^e,k, q^e)
= \frac{(k, k/a^e;q^e)_n}{(q^e, a^e q^e;q^e)_n} \\
\times
   \sum_{r=0}^{\lfloor n/d \rfloor} \frac{(a, q^{d/2} ;q^d)_r (1-aq^{2dr})
   (q^{-en}, k q^{en}; q^e)_{dr}}
   {(q^d, aq^{d/2};q^d)_r (1-a) (a^e q^{e(n+1)}, k^{-1}a^e q^{e(1-n)}; q^e)_{dr} }\\
   \times a^{(h-d+ed)r} k^{-dr} q^ {(h - d)dr^2 +
   (2e-1)dr/2}. \label{JSWPbeta}
\end{multline}
Thus each of $\Big( \wpalpha_n^{(d,e,h)}(a,k;q),
\wpbeta_n^{(d,e,h)}(a,k;q) \Big)$, $\Big(
\widetilde{\wpalpha}_n^{(d,e,h)}(a,k,q),$\\ $
\widetilde{\wpbeta}_n^{(d,e,h)}(a,k,q) \Big)$, and $\Big(
\bar{\wpalpha}_n^{(d,e,h)}(a,k,q), \bar{ \wpbeta}_n^{(d,e,h)}(a,k,q)
\Big)$  is a WP Bailey pair. Note that the series in each
of~\eqref{SWPbeta}--\eqref{JSWPbeta} may be expressed as a limiting
case of a very-well-poised ${}_{t+1}\phi_t$ basic hypergeometric
series, where \[ t = |2h-2d+1|+2ed +2,\] and as such is either summable
or transformable via standard formulas found in Gasper and Rahman's
book~\cite{GR04}.

\begin{proposition}
The WP multiparameter Bailey pairs \\$\Big(
\wpalpha_n^{(1,1,1)}(a,k,q),$$ \wpbeta_n^{(1,1,1)}(a,k,q) \Big)$, $\Big(
\widetilde{\wpalpha}_n^{(1,1,1)}(a,k,q),
\widetilde{\wpbeta}_n^{(1,1,1)}(a,k,q) \Big)$, and\\ $\Big(
\bar{\wpalpha}_n^{(1,1,1)}(a,k,q), \bar{ \wpbeta}_n^{(1,1,1)}(a,k,q)
\Big)$ are given by {\allowdisplaybreaks
\begin{align}
\wpalpha_n^{(1,1,1)} (a,k,q) &= \frac{ (-1)^n q^{n(n-1)/2} (1-aq^{2n}) (a;q)_n}{(1-a)(q;q)_n} \label{S111alpha}\\
\wpbeta_n^{(1,1,1)} (a,k,q) &   
=  \frac{ (-1)^n k^{n} q^{\binom n2} (k;q)_n } { a^n (q;q)_n } \label{S111beta}\\
\widetilde{\wpalpha}_n^{(1,1,1)} (a,k,q) &=
\frac{ (-1)^n a^{-n} (1-aq^{2n}) (a^2;q^2)_n}{(1-a)(q^2;q^2)_n} \label{E111alpha}\\
\widetilde{\wpbeta}_n^{(1,1,1)} (a,k,q) & 
 = \frac{(-1)^n a^{-n} (k^2;q^2)_n}{(q^2;q^2)_n}
\label{E111beta} \\
\bar{ \wpalpha}_n^{(1,1,1)} (a,k,q) &= \frac{  q^{-n/2} (1-aq^{2n}) (a, \sqrt{q}; q)_n}{(1-a)
(q, a\sqrt{q})_n} \label{JS111alpha} \\
\bar{\wpbeta}_n^{(1,1,1)} (a,k,q) & 
 = \frac{ (k, k\sqrt{q}/a;q)_n }{ (q, a\sqrt{q};q)_n } q^{-n/2}
\label{JS111beta}
\end{align}
}
\end{proposition}

\begin{proof}
Each of the $\alpha$'s is a direct substitution into the definition with $d=e=h=1$.
Each of the $\beta$'s follows from Jackson's summation of a very-well-poised
${}_6 \phi_5$~\cite[Eq. (II.21)]{GR04}.
\end{proof}

 We may now use these WP Bailey pairs to derive WP generalizations of Rogers-Ramanujan-Slater
type identities.
\begin{corollary}
\begin{gather}
 \sum_{n=0}^\infty  \frac{  (1-kq^{2n})  (k;q)_n   } {(1-k)  (q;q)_n  }
 (-1)^n k^{n} q^{n(3n-1)/2}= (kq;q)_\infty \label{S111WPWBL}
 \\
 \sum_{n=0}^\infty \frac{ (1-kq^{4n}) (-q,k;q^2)_n  }{ (1-k) (-kq, q^2;q^2)_n }
   (-1)^n k^n q^{2n^2-n} = \frac{ (kq^2;q^2)_\infty}{ (-kq;q^2)_\infty} \label{S111WPTBL}\\
 \sum_{n=0}^\infty \frac{ (1-kq^{2n}) (-1,k ;q)_n  }{ (1-k) (-kq, q;q)_n }
   (-1)^n k^n q^{n^2} = \frac{ (kq;q)_\infty}{ (-kq;q)_\infty}   \label{S111WPSSBL}\\
  \sum_{n=0}^\infty \frac{ (1-kq^{2n}) (k^2;q^2)_n }{  (1-k)(q^2;q^2)_n} (-1)^n q^{n^2} =
  {(kq;q)_\infty}{(q;q^2)_\infty} \label{E111WPWBL} \\
   \sum_{n=0}^\infty \frac{ (1-kq^{4n}) (-q;q^2)_n (k^2;q^4)_n }{  (1-k)(-kq;q^2)_n(q^4;q^4)_n}
   (-1)^n q^{n^2}=
  \frac{(kq^2;q^2)_\infty (q;q)_\infty }{(-kq;q^2)_\infty (q^4;q^4)_\infty}
    \label{E111WPTBL}\\
   \sum_{n=0}^\infty \frac{ (1-kq^{4n}) (k;q)_{2n} }{  (1-k) (q;q)_{2n} }
    q^{2n^2-n}=
  \frac{(kq^2;q^2)_\infty}{(q;q^2)_\infty}\label{JS111WPWBL}\\
   \sum_{n=0}^\infty \frac{ (1-kq^{4n}) (-q;q^2)_n (k;q)_{2n} }{  (1-k) (-kq;q^2)_n (q;q)_{2n} }
    q^{n^2-n}=
  \frac{(kq^2, -1 ;q^2)_\infty}{(-kq, q;q^2)_\infty} \label{JS111WPTBL}
 \end{gather}
\end{corollary}
\begin{proof} To obtain~\eqref{S111WPWBL},
insert~\eqref{S111alpha}--\eqref{S111beta} into~\eqref{wpeq} with  $a=1$ and $y,z\to\infty$.
To obtain~\eqref{S111WPTBL}, insert~\eqref{S111alpha}--\eqref{S111beta}
into~\eqref{wpeq} with $a=1$, $y=-\sqrt{q}$ and $z\to\infty$.
To obtain~\eqref{S111WPSSBL}, insert~\eqref{S111alpha}--\eqref{S111beta}
into~\eqref{wpeq} with $a=1$, $y=-{q}$ and $z\to\infty$.
 To obtain~\eqref{E111WPWBL},
insert~\eqref{E111alpha}--\eqref{E111beta} into~\eqref{wpeq} with  $a=1$ and $y,z\to\infty$.
To obtain~\eqref{E111WPTBL}, insert~\eqref{E111alpha}--\eqref{E111beta}
into~\eqref{wpeq} with $a=1$, $y=-\sqrt{q}$ and $z\to\infty$.
 To obtain~\eqref{JS111WPWBL},
insert~\eqref{JS111alpha}--\eqref{JS111beta} into~\eqref{wpeq} with  $a=1$ and $y,z\to\infty$.
To obtain~\eqref{JS111WPTBL}, insert~\eqref{JS111alpha}--\eqref{JS111beta}
into~\eqref{wpeq} with $a=1$, $y=-\sqrt{q}$ and $z\to\infty$.
\end{proof}

\begin{remark}
Setting $k=0$ in~\eqref{E111WPWBL} recovers Eq. (3) of
Slater~\cite{S52}. We had obtained~\eqref{JS111WPWBL}
and~\eqref{JS111WPTBL} earlier via another method (see
Eqs.~\eqref{q1/2pairpr} and~\eqref{q1/2pairpr3}). Setting $k=0$
in~\eqref{JS111WPWBL} recovers Eq. (9) of Slater~\cite{S52}, an
identity originally due to Jackson~\cite[p. 179, 3 lines from
bottom]{J28}. Note that these identities may also be derived as
special cases of \eqref{6phi5eq}.

\end{remark}

If any of $d$, $e$, or $h$ is greater than $1$, then the
representation of the $\beta$ as a finite product times a very-well-poised
${}_{t+1}\phi_t$ will have $t>6$, will thus not be summable.
Accordingly, the WP-Rogers-Ramanujan-Slater type identities
obtained from these will involve double sums.

Let us now consider the case which leads to a generalization of the
first Rogers-Ramanujan identity.

\begin{proposition}
$\Big( \wpalpha_n^{(1,1,2)}(a,k, q), \wpbeta_n^{(1,1,2)}(a,k,q)
\Big)$ is given by {\allowdisplaybreaks
\begin{align}
\wpalpha_n^{(1,1,2)} (a,k,q) &= \frac{ (-1)^n a^n q^{n(3n-1)/2} (1-aq^{2n}) (a;q)_n}{(1-a)(q;q)_n} \label{S112alpha}\\
\wpbeta_n^{(1,1,2)} (a,k,q) &
=  (k;q)_n \sum_{j=0}^n \frac{ (-1)^j k^j q^{\binom j2+nj} (k/a;q)_{n-j} }{(q;q)_j (q;q)_{n-j}  } \label{S112beta}
\end{align}
}
\end{proposition}

\begin{proof}
The $\alpha_n^{(1,1,2)}(a,k,q)$ follows by direct substitution
into~\eqref{SWPalpha}. The $\beta_n^{(1,1,2)}$ $(a,k,q)$ follows by
specializing~\eqref{SWPbeta} and applying Watson's $q$-analog of
Whipple's theorem~\cite[Eq. (III.18)]{GR04}.
\end{proof}

\begin{corollary}[a WP-generalization of the first Rogers-Ramanujan
identity]\label{crr}
\begin{equation*}
\sum_{n,j\geq 0} \frac{  (1-kq^{2n+2j}) q^{n^2 + 3nj + j(5j-1)/2} (-1)^j k^j (k;q)_n (k;q)_{n+j}}{(1-k)
(q;q)_j (q;q)_n }
= \frac{ (kq;q)_\infty }{ (q, q^4; q^5)_\infty }.
\end{equation*}
\end{corollary}

\begin{proof}
Insert $\Big( \wpalpha_n^{(1,1,2)}(a,k,q), \wpbeta_n^{(1,1,2)}(a,k,q) \Big)$ into~\eqref{wpeq}
with $a=1$ and $y,z\to\infty$, interchange the order of summation on the left hand side and
apply Jacobi's triple product identity~\eqref{JTP} on the right hand side.
\end{proof}

\begin{corollary}[a WP-generalization of the first G\"ollnitz-Gordon
identity]\label{cgg}
\begin{multline} \label{WPGG}
\sum_{n,j\geq 0} \frac{  (-k)^j(1-kq^{4n+4j}) q^{(n+2j)^2 -j}
(-q;q^2)_{n+j} (k;q^2)_n
(k;q^2)_{n+j}}{ (1-k) (-kq;q^2)_{n+j} (q^2;q^2)_j (q^2;q^2)_n }\\
= \frac{ (kq^2;q^2)_\infty }{ (-kq;q^2)_\infty (q, q^4, q^7; q^8)_\infty }.
\end{multline}
\end{corollary}

\begin{proof}
Insert $\Big( \wpalpha_n^{(1,1,2)}(a,k,q), \wpbeta_n^{(1,1,2)}(a,k,q) \Big)$ into~\eqref{wpeq}
with $a=1$, $y=-\sqrt{q}$ and $z\to\infty$, interchange the order of summation on the left hand side and
apply Jacobi's triple product identity~\eqref{JTP} on the right hand side.  Finally, replace $q$ by
$q^2$ throughout.
\end{proof}

\begin{remark}
By sending $k\to 0$ in~\eqref{WPGG}, we recover Identity (36) of
Slater \cite{S52}. An equivalent analytic identity was recorded by
Ramanujan in the lost notebook~\cite{AB07}. This identity became
well-known after being interpreted partition theoretically by
G\"ollnitz~\cite{G60} and Gordon~\cite{G65}.
\end{remark}

Many additional WP-analogs of known Rogers-Ramanujan type identities
could easily be derived using the
WP-multiparameter Bailey pairs.  We shall content ourselves here with several examples
where the series expressions are not too complicated.  Each of the following identities can
be proved by inserting an appropriate WP-Bailey pair into a limiting case of~\eqref{wpeq},
and applying Jacobi's triple product identity~\eqref{JTP}.

A WP-generalization of the first Rogers-Selberg mod 7 identity~\cite[p. 339]{R94}; cf.~\cite[Eq. (33)]{S52}:
\begin{multline}
\sum_{n,r\geq 0} \frac{ (1-kq^{4n+4r}) (k;q^2)_{n+2r} }{ (1-k) (q^2;q^2)_r (-q;q)_{2r} (q^2;q^2)_n}
(-1)^n k^n q^{n(3n-1) + 4nr + 2r^2} \\
= \frac{ (kq^2;q^2)_\infty (q^3,q^4,q^7; q^7)_\infty}{ (q^2;q^2)_\infty}.
\end{multline}

A WP-generalization of the Jackson-Slater identity~\cite[p. 170, 5th Eq.]{J28}; cf.~\cite[Eq. (39)]{S52}:
\begin{multline}
\sum_{n,r\geq 0} \frac{ (1-kq^{4n+4r}) (k;q^2)_{n+2r} (kq;q^2)_{n+r} (q;q^2)_r }{ (1-k) (q^2;q^2)_r
(q;q^2)_{n+r} (kq;q^2)_r (q^2;q^2)_n}
(-1)^r  q^{2n^2-n+2nr+r^2} \\
= \frac{ (kq^2;q^2)_\infty (-q^3,-q^5,q^8; q^8)_\infty}{ (q^2;q^2)_\infty}.
\end{multline}

A WP-generalization of Bailey's mod 9 identity~\cite[p. 422, Eq. (1.8)]{B47}, cf.~\cite[Eq. (42)]{S52}:
\begin{multline}
\sum_{n, r\geq 0} \frac{(1-kq^{6n+6r}) (k;q^3)_{n+2r}(q;q)_{3r} }
{(1-k)(q^3;q^3)_n (q^3;q^3)_{2r} (q^3;q^3)_r} (-1)^n k^n q^{3n(3n-1)/2 + 6nr + 3r^2}
\\ = \frac{(kq^3;q^3)_\infty (q^4,q^5,q^9;q^9)_\infty }{ (q^3;q^3)_\infty }.
\end{multline}

A WP-generalization of Rogers's mod 14 identity~\cite[p. 341, Ex. 2]{R94}; cf.~\cite[Eq. (61)]{S52}:
\begin{multline}
\sum_{n,r\geq 0} \frac{ (1-kq^{2n+2r}) (k;q)_{n+2r} }
{ (1-k) (q;q^2)_r (q;q)_r (q;q)_n }  (-1)^n k^n q^{n(3n-1)/2 + 2nr + r^2} \\
= \frac{(kq;q)_\infty (q^6,q^8,q^{14};q^{14})_\infty}{ (q;q)_\infty}.
\end{multline}

Remark: All the of the double-sum identities above, and those in
Corollaries \ref{crr} and \ref{cgg}, are new.

\section{Slater Revisited}
It would interesting to lift all the Bailey pairs found by Slater to
WP-Bailey pairs, but at present we do not have a general method that
will allow us to do this.

 As was noted earlier, it is likely that
finding lifts of the other Bailey pairs will be more difficult, as
experimentation suggests that the sequence $\alpha_n$ will be
dependent on the parameter $k$.

It is hoped that some of the results in the present paper might
interest others in the search for lifts of the remaining Bailey
pairs in the Slater papers \cite{S51, S52}.

 \allowdisplaybreaks{

}

\end{document}